\documentclass[a4paper]{scrartcl}

\usepackage[english]{babel} 
\usepackage[utf8]{inputenc} 

\usepackage[authoryear]{natbib}
\usepackage{amsmath} 
\usepackage{amsthm} 
\usepackage{amsfonts} 
\usepackage{bbm}
\usepackage{mathtools}
\usepackage{enumerate} 
\usepackage{xcolor}
\usepackage{listings} 
\usepackage[algosection]{algorithm2e}
\RestyleAlgo{ruled}
\SetKwInput{Input}{Input}
\SetKwInput{Output}{Output}
\usepackage{graphicx}
\usepackage{subcaption}
\usepackage{float}
\usepackage[hidelinks]{hyperref} 
\graphicspath{ {./Plots/} }

\usepackage{tikz}
\usepackage{comment}
\usetikzlibrary{automata, positioning,arrows}
\usepackage{booktabs} 
\usepackage{array}    
\usepackage[misc]{ifsym}

\title{To report or not to report:\\
Optimal claim reporting in a bonus-malus system}

\date{\normalsize January 2026}
\author{Lea Enzi and Stefan Thonhauser\thanks{Institute of Statistics, Graz University of Technology.\\\Letter~stefan.thonhauser@tugraz.at} }

\newtheorem{thm}{Theorem}[section]

\newtheorem{lemma}{Lemma}[section]
\newtheorem{corollary}{Corollary}[section]
\theoremstyle{remark}
\newtheorem{remark}{Remark}[section]

\theoremstyle{definition}

\numberwithin{equation}{section}

\newcommand*\diff{\mathop{}\!\mathrm{d}}

\newcommand{\ind}{\mathbbm{1}_}

\definecolor{blue}{rgb}{0.0, 0.53, 0.74}
 
\begin{document}
\maketitle	

\section*{Abstract}
We study an optimal claim reporting problem in a bonus-malus setting. We assume, that the insurance contract consists of two regimes, where reporting a claim leads to a transition to a higher-premium regime, whereas remaining claim-free for a prespecified time period results in a shift to the lower premium regime. The insured can decide whether or not to report an occurred claim.
We formulate this as an optimal control problem, where the policyholder follows a barrier-type reporting strategy, with the goal of maximizing the expected value of a function of their terminal wealth.
We show that the associated value function is the unique viscosity solution to a system of Hamilton-Jacobi-Bellman equations. This characterization allows us to compute numerical approximations of the optimal barrier strategies.

\section{Introduction}
The bonus-malus system is a pricing method commonly used in car insurance and is designed to adjust premiums based on a policyholder's claims history. Its structure is typically modeled as a discrete-state system, where each state or class corresponds to a specific premium level. Transitions between classes depend on the number of claims incurred: policyholders who remain claim-free within a defined period move toward lower-premium classes, while reporting a claim results in a shift to a higher-premium class. An overview on bonus-malus systems is given by \cite{Lemaire.1998}, who also provides a framework for the implementation in actuarial science.

Actually, being “claim-free” does not necessarily imply the complete absence of loss events. In practice, policyholders may choose not to report certain claims, especially when, from their perspective, the premium increase from entering a higher class exceeds potential future losses. Such strategic non-reporting preserves the current class and is often referred to as ``hunger for bonuses''. 

As part of the decision-making process, barrier strategies can serve as a tool to decide if a claim should be reported. This means, that the policyholder only reports a claim if its value exceeds a certain barrier level. While \cite{Straub} analyzes constant barrier strategies that remain fixed within a period, \cite{Lanzenauer1974} motivates the use of optimal barrier functions depending on the current state, the number of already reported claims and when the decision is to be made throughout the year. Whether a particular reporting strategy corresponding to a barrier is optimal, depends on the choice of the underlying objective functional, such as minimizing expected future costs, maximizing wealth, etc. 

In the literature on optimal claim reporting, two main approaches exist: in a discrete-time model, a switch to a new class occurs at the end of a specified period (e.g., one year), whereas in a continuous time model, the change happens immediately upon a claim being reported.  A discrete time setting based on a Markov chain model is, for example, discussed by  \cite{charpentier2017}.
\cite{DePril1979} extends the setting of \cite{Lanzenauer1974} to the continuous case. \cite{ZacksLevikson2004} analyze a discrete and a continuous time framework, where they search for a constant optimal barrier, depending only on the state. A frequently adopted assumption in the literature, reflecting real-world practice, is that the policyholder transitions to a better class if they remain claim-free for a specified time.  

In a continuous time setting with two premium relevant states and a random time horizon, \cite{ Young2025_meanvariance,Young2025} find an optimal barrier by finding solutions to corresponding differential equations. They assume that the insurance holder immediately transitions to a better class whenever an occurring claim is not reported. This implies that the insurer has full information about all (reported and not reported) claims.  

In this contribution, we consider a continuous model with two possible states and assume a finite time horizon, i.e., the contract terminates at some time $T>0$. The policyholder transitions to the better class, whenever they do not report a claim for a predetermined period. The objective is to find an optimal barrier type strategy, with the goal of maximizing the expected value of a function of their wealth at maturity. To solve this, we apply stochastic control theory. Our approach combines elements from aforementioned references, in particular by extending the optimal control framework to include the transition to a lower premium class when claim-free.

The outline of this chapter is as follows: in Section \ref{sec:ModelSetup}, we define the underlying model and specify the control problem. In Section \ref{sec:valueFunction}, we show regularity of the value function. In Section \ref{sec:HJB}, we show that the value function is a viscosity solution to the corresponding system of Hamilton-Jacobi-Bellman equations. In Section \ref{sec:Numerics}, we solve the equation numerically and find a corresponding Markovian control.

\section{Model setup}
\label{sec:ModelSetup}
We consider an insurance buyer whose aggregated claim process is given by a compound Poisson process $Z=(Z_t)_{t\geq 0}$, i.e.,
\begin{equation*}
    Z_t = \sum_{j=1}^{N_t}Y_j.
\end{equation*}
Here, $N=(N_t)_{t\geq 0}$ is a Poisson process with intensity $\lambda>0$ and jump times $\{T_j\}_{j \in \mathbb{N}}$, and $\{Y_j\}_{j \in \mathbb{N}}$ are i.i.d.~random variables with continuous distribution function $F_Y$ with $F_Y(0)=0$, independent of $N$. Further, we assume that $\mathbb{E}[Y^2]<\infty$.

The insurance contract distinguishes between two insurance classes:
\[
\begin{aligned}
C_1 &\text{ with premium function } \pi_1, \\
C_2 &\text{ with premium function } \pi_2,
\end{aligned}
\]
where $\pi_1$ and $\pi_2$ are specified later. We assume that the contract terminates at some time $T>0$. 
Furthermore, we assume that if the policyholder is in $C_1$ and reports a claim, they are immediately reclassified to $C_2$ and charged a higher premium, meaning $\pi_2>\pi_1$. If they are already in class $C_2$ and no claim gets reported (or no claim occurs), they transition to $C_1$ after some time $\mathcal{S}>0$. On the other hand, the reporting of a claim results in a reset of the duration spent in $C_2$.

Furthermore, we assume that the insurance policy includes the possibility of a deductible $m_i$, which is modeled by a retention function $r$, given by 
\begin{equation*}
    r(y,m_i)=\min(y,m_i).
\end{equation*}

The policyholder wants to find an optimal claim reporting strategy, which maximizes their terminal wealth (or a function thereof). The control is implemented via a barrier strategy $b=(b_t)_{t\geq 0}$, which is a predictable process with respect to $( \mathcal{F}_t^Z )_{t \geq 0}$ which takes values in $\mathbb{R}^+_0$. This means that if a claim occurs at time $t$, the policyholder only reports it, if it exceeds $b_t$. We denote the set of all such admissible barrier strategies on $[t,T]$ by $\mathcal{B}(t)$ for all $0\leq t < T$. 

To track the current state the process is in, we introduce a process $I=(I_t)_{t \geq 0}$, where $I_t=k$ indicates that the policyholder is in class $C_k$ at time $t$, for $k \in \{1,2\}$. We suppress the dependence of $I$ on $b$ for notational convenience.

The controlled process $S^b=(S_t^b)_{t\geq 0}$, which tracks the time since the insurance buyer last reported a claim while remaining in the same insurance class, is given by
\begin{equation*}
\diff S_t^b=\diff t - S_{t-}^b \ind{\{Y_{N_t} > b_t\}}\diff N_t - S_{t-}^b \ind{\{S_{t-}^b=\mathcal{S}\}} \ind{\{I_{t-}=2\}}.
\end{equation*} 
The process is set to zero whenever a claim is reported, and it is also set to zero when it reaches $\mathcal{S}$ while the policyholder is in class 2. The latter corresponds to not reporting a claim for a sufficiently long period, resulting in a switch to class 1.

To allow for a premium reduction or bonus when remaining claim free, we include this in the premium functions $\pi_i(s,m_i)$, which we now assume to depend on the time since entering $C_i$ and the deductible $m_i \geq 0$. 

We say that insured's income rate is given by a constant $c$ which is larger than the premium rate. Then, the wealth process $X^b=(X_t^b)_{t\geq0}$ evolves according to the following dynamic:
\begin{equation*}
     \diff X^b_t= (c-\pi_{I_{t-}}(S_{t-}^b,m_{I_{t-}}))\diff t - \left(Y_{N_t} \ind{\{Y_{N_t} \leq b_t\}}+ r(Y_{N_t},m_{I_{t-}}) \ind{\{Y_{N_t} > b_t\}}\right) \diff N_t,
     \end{equation*}
with some initial wealth $x \in \mathbb{R}$. 

For some fixed $b \geq 0$, the underlying model $(I_t,t,S_t^b,X_t^b)$ is a \textit{piecewise deterministic Markov process} with two external states (classes) and an active boundary in the second state.

For the generator of the wealth process, we consider a function 
\begin{equation*}
    f = (f_1,f_2) \in \mathcal{D}(\mathcal{A}), \qquad f(i,t,s,x)=f_i(t,s,x),
\end{equation*} where $\mathcal{D}(\mathcal{A})$ is the domain of the generator. For $f$ to be in the domain of the generator, it must, in particular, be absolutely continuous along the deterministic flow. For more information on the domain of a generator, we refer to \citep{Davis} or \citep{Rolski}. 

Since the process has two possible states representing the current class the insured person is in, the generator acts component-wise, leading to a system of coupled equations:
\begin{align*}
    \mathcal{A}^b f_1(t,s,x)= &\left(\frac{\partial}{\partial t}+\frac{\partial}{\partial s}\right) f_1(t,s,x)+\left(c-\pi_1(s,m_1)\right)\frac{\partial}{\partial x}f_1(t,s,x) \nonumber \\
    &+ \lambda \int_{0}^{b}f_1(t,s,x-y)\diff F_Y(y) \nonumber \\
    &+ \lambda \int_{b}^{\infty}f_2(t,0,x-r(y,m_1))\diff F_Y(y) - \lambda f_1(t,s,x),\\
    \mathcal{A}^bf_2(t,s,x) = &\left(\frac{\partial}{\partial t}+\frac{\partial}{\partial s}\right) f_2(t,s,x)+\left(c-\pi_2(s,m_2)\right)\frac{\partial}{\partial x}f_2(t,s,x) \nonumber \\
     &+ \lambda \int_{0}^{b}f_2(t,s,x-y)\diff F_Y(y)  \nonumber \\
     &+ \lambda \int_{b}^{\infty}f_2(t,0,x-r(y,m_2))\diff F_Y(y)- \lambda f_2(t,s,x),
\end{align*}
with boundary condition
\begin{align*}
    f_2(t,\mathcal{S},x)=f_1(t,0,x), \quad t \in [\mathcal{S},T].
\end{align*}

The policyholder now aims to maximize their terminal wealth. We consider the functional
\begin{equation*}
    J(i,t,s,x,b)=\mathbb{E}_{i,t,s,x}[h(X^b_T)]=\mathbb{E}[h(X^b_T)|I_t=i,S_t^b=s,X_t^b=x],
\end{equation*}
where $i \in \{1,2\}$, $0\leq s \leq t \leq T$, $s \leq \mathcal{S}$ if $i=2$ and $x \in \mathbb{R}$. 
Further, $h:\mathbb{R}\rightarrow\mathbb{R}$ is a Lipschitz continuous function and $b \in \mathcal{B}(t)$ is an admissible barrier strategy. At time $T$, it holds that
\begin{equation*}
    J(i,T,s,x,b)=h(x).
\end{equation*}
We want to find 
\begin{align*}
    V(i,t,s,x)=\sup_{b \in \mathcal{B}(t)}J(i,t,s,x,b).
\end{align*}
As an explicit optimal control cannot be determined, we focus in the subsequent sections on the regularity and properties of the value function, thereby providing a basis for numerical computations.
\section{Regularity of the value function}
\label{sec:valueFunction}
To proceed with the theoretical results, we need a certain regularity of the value function. In particular, it should belong to the domain of the generator $\mathcal{D}(\mathcal{A)}$. We show continuity in the following lemma.
\begin{lemma}
    The map $(t,s,x) \mapsto V(i,t,s,x)$ is Lipschitz continuous for $i \in \{1,2\}$.
\end{lemma}

\begin{proof}
    We start by fixing $i,t,s$ and consider $x_1,x_2 \in \mathbb{R}$. For some $\varepsilon$, let $b_\varepsilon^{x_1} \in \mathcal{B}(t)$ be an $\varepsilon$-optimal strategy corresponding to $x_1$. Then, this control is also an admissible control of the process starting in $x_2$. Therefore, the processes move in parallel until time $T$, i.e.,
    \begin{align*}
        V(i,t,s,x_1)-V(i,t,s,x_2)   &\leq L_h \mathbb{E} \left[\big |X_T^{x_1,b_\varepsilon^{x_1}} - X_T^{x_2,b_\varepsilon^{x_1}}\big|\right]+\varepsilon\\
                                &\leq L_h |x_1-x_2|+\varepsilon,
    \end{align*}
    where $L_h$ is the Lipschitz constant of the function $h$. Equivalently, we can bound $V(i,t,s,x_2)-V(i,t,s,x_1)$ and therefore 
    \begin{equation*}
        |V(i,t,s,x_1)-V(i,t,s,x_2)| \leq L_h |x_1-x_2|+\varepsilon.
    \end{equation*}\\
    Now, we fix $t,x$ and consider $s_1,s_2 \in \mathbb{R}^+$ with $s_1, s_2 \leq t$. We start with the case $i=1$. We choose an $\varepsilon$-optimal control $b^{\varepsilon, s_1}$ corresponding to $V(1,t,s_1,x)$. It is impossible to switch from state 1 to state 2 unless a (high enough) claim occurs. As soon as they transition to state 2, the time spent in the state gets set to zero. Let 
    \begin{equation*}
      \tau =  \inf_{i \in \mathbb{N}}\left\{T_i:T_i \geq t, Y_i>b_{T_i}^{\varepsilon, s_1}\right\}. 
    \end{equation*}
    
    Then,
    \begin{align}
     \label{eq:continuity_in_s}
        &\mathbb{E}\left[|X_T^{s_1,b^{\varepsilon, s_1}} - X_T^{s_2,b^{\varepsilon, s_1}}| \right] \nonumber\\
        &\quad= \mathbb{E}\left[\left| \int_{s_1}^{s_1+(\tau \wedge T)-t} (c-\pi_1(s,m_1))\diff s - \int_{s_2}^{s_2+(\tau \wedge T)-t} (c-\pi_1(s,m_1))\diff s\right|\right] \nonumber\\
        &\quad =  \mathbb{E}\left[\biggl|\int_{s_1 \wedge s_2}^{s_1 \vee s_2} (c-\pi_1(s,m_1)) \diff s - \int_{s_1\wedge s_2+(\tau \wedge T)-t}^{s_1 \vee s_2+(\tau \wedge T)-t} (c-\pi_1(s,m_1)) \diff s \biggr|\right] \nonumber \\
        &\quad \leq 2c|s_1-s_2|.
    \end{align}
    Therefore,
    \begin{align*}
        V(1,t,s_1,x)-V(1,t,s_2,x) \leq 2 L_h c |s_1-s_2|+\varepsilon.
    \end{align*}
    Again, a similar bound holds for $V(1,t,s_2,x)-V(1,t,s_1,x)$ and therefore,
    \begin{equation*}
        |V(1,t,s_1,x)-V(1,t,s_2,x)| \leq 2 L_h c  |s_1-s_2|+\varepsilon.
    \end{equation*}
    For $i=2$, not reporting a claim for a total time $\mathcal{S}$ leads to a deterministic transition to class $C_1$; hence, due to the presence of an active boundary, this case requires more careful consideration.
    Let again $b_\varepsilon^{s_1}$ be an $\varepsilon$-optimal control corresponding to $V(1,t,s_1,x)$. Further, let $s_1 \geq s_2$ w.l.o.g. and $\tau$ be defined as before. We consider three different cases:
    \begin{enumerate}
        \item 
        Let $\mathcal{E}_1=\{\omega \in \Omega|s_1 + (\tau(\omega) \wedge T)-t \leq \mathcal{S}\}$.
        Then,
        \begin{align*}
            &\mathbb{E}\left[\ind{\mathcal{E}_1}|X_T^{s_1,b_\varepsilon^{s_1}} - X_T^{s_2,b_\varepsilon^{s_1}}| \right] \\
            & \quad = \mathbb{E}\left[\ind{\mathcal{E}_1}\left|\int_{s_1}^{s_1 + \tau\wedge T - t}(c- \pi_2(s,m_2))\diff s - \int_{s_2}^{s_2 + \tau\wedge T - t}(c- \pi_2(s,m_2))\diff s\right|\right]\\
            &\quad \leq 2 L_h c|s_1-s_2|,
        \end{align*}
        by calculations similar to \eqref{eq:continuity_in_s}.
        \item 
        Let $\mathcal{E}_2=\{\omega \in \Omega|s_2 + (\tau(\omega) \wedge T)-t \geq \mathcal{S}\}$. Then,
         \begin{align*}
            &\mathbb{E}\left[\ind{\mathcal{E}_2}|X_T^{s_1,b_\varepsilon^{s_1}} - X_T^{s_2,b_\varepsilon^{s_1}}|\right] \leq \left|\int_{s_1}^{\mathcal{S}}(c- \pi_2(s,m_2))\diff s - \int_{s_2}^{\mathcal{S}}(c- \pi_2(s,m_2))\diff s\right|\\
            &+ \mathbb{E}\left[\ind{\mathcal{E}_2}\left|\int_{0}^{\tau \wedge T-t-(\mathcal{S}-s_1)}(c- \pi_1(s,m_1))\diff s - \int_{0}^{\tau \wedge T-t-(\mathcal{S}-s_2)}(c- \pi_1(s,m_1))\diff s\right|\right]\\
            &= \left|\int_{s_2}^{s_1}c-\pi_2(s,m_2)\diff s\right|
            + \mathbb{E}\left[\ind{\mathcal{E}_2}\left|\int_{\tau \wedge T-t-(\mathcal{S}-s_2)}^{\tau \wedge T-t-(\mathcal{S}-s_1)}(c- \pi_1(s,m_1))\diff s\right|\right]\\
         &  \leq 2c|s_1-s_2|.
        \end{align*}
        
        \item 
        Let $\mathcal{E}_3=\{\omega \in \Omega|s_2 + (\tau(\omega) \wedge T)-t \leq \mathcal{S} \leq s_1 + (\tau(\omega) \wedge T)-t\}$.
            We consider
            \begin{align*}
                &\mathbb{E}\left[ |X_T^{s_1,b_\varepsilon^{s_1}} - X_T^{s_2,b_\varepsilon^{s_1}}|\ind{\mathcal{E}_3}\right]=
                \mathbb{E}\left[ |X_T^{s_1,b_\varepsilon^{s_1}} - X_T^{s_2,b_\varepsilon^{s_1}}|\ind{\{\mathcal{S}-s_1\leq \tau\wedge T-t \leq \mathcal{S}-s_2\}}\right] \\
                &=\mathbb{E}\left[ |X_T^{s_1,b_\varepsilon^{s_1}} - X_T^{s_2,b_\varepsilon^{s_1}}|\ind{\{\mathcal{S}-s_1\leq \tau-t \leq \mathcal{S}-s_2\}}\ind{\{\tau \leq T\}}\right] \\
                &\quad + \mathbb{E}\left[ |X_T^{s_1,b_\varepsilon^{s_1}} - X_T^{s_2,b_\varepsilon^{s_1}}|\ind{\{\mathcal{S}-s_1\leq T-t \leq \mathcal{S}-s_2\}}\ind{\{\tau > T\}}\right]\\
            \end{align*}
            We start with the first summand.
            \begin{align*}
                &\mathbb{E}\left[ |X_T^{s_1,b_\varepsilon^{s_1}} - X_T^{s_2,b_\varepsilon^{s_1}}|\ind{\{\mathcal{S}-s_1\leq \tau-t \leq \mathcal{S}-s_2\}}\ind{\{\tau \leq T\}}\right] \\
                &= \mathbb{E}\left[ |X_T^{s_1,b_\varepsilon^{s_1}} - X_T^{s_2,b_\varepsilon^{s_1}}|\bigg \vert \mathcal{S}-s_1\leq \tau-t \leq (\mathcal{S}-s_2) \wedge (T-t)\right]\\
                &\hfill \mathbb{P}(\mathcal{S}-s_1\leq \tau-t \leq  (\mathcal{S}-s_2) \wedge (T-t))
            \end{align*}

            Since the expectation can be bounded in terms of the maximal drift and the expected jump size together with the expected number of jumps until time $T$, we focus on the probability.
            \begin{align*}
                &\mathbb{P}(\mathcal{S}-s_1\leq \tau-t \leq  (\mathcal{S}-s_2) \wedge (T-t)) 
                \\& \qquad= \sum_{i=1}^{\infty} \mathbb{P}(\mathcal{S}-s_1\leq \tau -t\leq  (\mathcal{S}-s_2) \wedge (T-t)|\tau=T_i)\mathbb{P}(\tau=T_i)\\
                & \qquad \leq \sum_{i=1}^{\infty}\mathbb{P}(\mathcal{S}-s_1\leq T_i-t \leq  (\mathcal{S}-s_2) \wedge (T-t))\\
                & \qquad = \sum_{i=1}^{\infty} \int_{\mathcal{S}-s_1}^{ (\mathcal{S}-s_2) \wedge (T-t)} \frac{u^{i-1}\lambda^i e^{-\lambda u}}{\Gamma(i)}\diff u,\\
                & \qquad = \int_{\mathcal{S}-s_1}^{ (\mathcal{S}-s_2) \wedge (T-t)} \sum_{i=1}^{\infty}\frac{u^{i-1}\lambda^i e^{-\lambda u}}{\Gamma(i)} \diff u\\
                & \qquad = \lambda( (\mathcal{S}-s_2) \wedge (T-t)-(\mathcal{S}-s_1))\\
                & \qquad \leq \lambda(s_1-s_2),
            \end{align*}
            by using that the sum of i.i.d.~exponentially distributed random variables is Gamma distributed and Tonelli's theorem. \\
            Now, we consider the second summand.
            \begin{align*}
                 &\mathbb{E}\left[ |X_T^{s_1,b_\varepsilon^{s_1}} - X_T^{s_2,b_\varepsilon^{s_1}}|\ind{\{\mathcal{S}-s_1\leq T-t \leq \mathcal{S}-s_2\}}\ind{\{\tau > T\}}\right]\\
                 &\leq \bigg|\int_{s_1}^\mathcal{S} (c-\pi_2(s,m_2))\diff s+ \int_0^{(T-t)-(\mathcal{S}-s_1)}(c-\pi_1(s,m_1))\diff s \\
                & \quad- \int_{s_2}^{s_2 + (T-t)} (c-\pi_2(s,m_2)) \diff s \bigg|\\
                &= \bigg| \int_{s_2 + (T-t)}^{\mathcal{S}} (c-\pi_2(s,m_2)) \diff s - \int_{s_2}^{s_1} (c-\pi_2(s,m_2)) \diff s\\
                &+ \int_0^{(T-t)-(\mathcal{S}-s_1)}(c-\pi_1(s,m_1))\diff s \bigg| \\
                &\leq 2c\left(s_1-s_2 \right).
            \end{align*}
    \end{enumerate}
Therefore, we can find a constant $\tilde{K}>0$ such that
\begin{equation*}
    V(2,t,s_1,x)-V(2,t,s_2,x) \leq \tilde{K}|s_1-s_2|+\varepsilon.
\end{equation*}

Again, something similar holds for $V(2,t,s_2,x)-V(2,t,s_1,x)$ and therefore we can find a constant $K>0$, 
\begin{equation*}
    |V(2,t,s_1,x)-V(2,t,s_2,x)| \leq K|s_1-s_2|+\varepsilon.
\end{equation*}

Left to show is the case where we fix $i,s,x$ and consider times $t_1<t_2$. By choosing an $\varepsilon$-optimal control $b^\varepsilon \in \mathcal{B}(t_1)$. We define $\mathcal{E}=\{w \in \Omega | T_1(\omega)>t+(t_2-t_1)\}$. Then,
    \begin{align*}
        V(i,t_1,s,x)-V(i,t_2,s,x) \leq L_h \mathbb{E}\biggl[\ind{\mathcal{E}} \int_s^{s+(t_2-t_1)} (c-\pi_{i}(v,m_{i})) \diff v\\ + \ind{\mathcal{E}^c} (cT+\sum_{i=1}^{N_{(T-t_2)}} Y_i+r(Y_i,m_{I_{\tilde{T}_{i}-}}))\biggr] + \varepsilon\\
        \leq L_h c(t_2-t_1) \mathbb{P}(\mathcal{E}) + \mathbb{P}(\mathcal{E}^c)(cT+2\lambda T \mu)+\varepsilon.
    \end{align*}
Since $\mathbb{P}(\mathcal{E}) \leq 1 $ and  $\mathbb{P}(\mathcal{E}^c)=\mathbb{P}(T_1-t\leq(t_2-t_1))=1-e^{-\frac{(t_2-t_1)}{\mu}}\rightarrow 0$ for $t_2\rightarrow t_1$, the whole term vanishes for $t_2\rightarrow t_1$. On the other hand, let $\overline{b}^{\varepsilon} \in \mathcal{B}(t_2)$ be an $\varepsilon$-optimal control and $b \in \mathcal{B}(t_1)$ an arbitrary control. Then, define $\tilde{b}_t=b_t$ for $t < t_2$ and $\tilde{b}_t=\overline{b}^{\varepsilon}_t \ind{\{T_1\geq t_2\}}$+ $b_t \ind{\{T_1< t_2\}}$ for $t \geq t_2$. Now, similar to before
\begin{align*}
       V(i,t_1,s,x)-V(i,t_2,s,x) \leq L_h c(t_2-t_1) + \mathbb{P}(\mathcal{E}^c)(cT+2\lambda T \mu)+\varepsilon,
    \end{align*}
which again tends to zero if $t_2 \rightarrow t_1$.\\
Altogether, there is a constant $M$ such that
\begin{align*}
    &|V(i,t_1,s_1,x_1)-V(i,t_2,s_2,x_2)|\leq |V(i,t_1,s_1,x)-V(i,t_2,s_1,x)|\\&\qquad +|V(i,t_2,s_1,x)-V(i,t_2,s_2,x)|+|V(i,t_2,s_2,x_1)-V(i,t_2,s_2,x_2)|\\&\leq M \Vert(t_1,s_1,s_1)-(t_2,s_2,x_2)\Vert_1 + 3\varepsilon,
\end{align*}
where $\Vert\cdot\Vert_1 $ denotes the classical $\ell_1$-norm.
The function is Lipschitz continuous because $\varepsilon$ is arbitrary and each difference can be bounded independently of the other fixed variables.
\end{proof}

Given the continuity of the value function, it follows by standard arguments that a dynamic programming principle is satisfied --- see e.g. \citep{FlemingSoner2006} for general theory on controlled Markov processes, or \citep{Schmidli2008} in an insurance context.

\begin{corollary}[Dynamic Programming Principle]
    \label{cor:DPP}
    For any bounded stopping time $\tau$, it holds that
    \begin{align}
        V(i,t,s,x)= \sup_{b \in \mathcal{B}(t)} \mathbb{E}_{i,t,s,x}\left[ V(I_\tau,\tau,S_\tau^b,X_\tau^b) \ind{\{T>\tau\}} + h(X_T^b) \ind{\{T\leq\tau\}}\right].
    \end{align}
\end{corollary}

\newpage
\section{Characterization via a system of HJB equations}
\label{sec:HJB}
The dynamic programming principle allows us to connect the value function to a system of Hamilton-Jacobi-Bellman equations. By verifying that $V$ is a viscosity solution to this system, a comparison theorem ensures its uniqueness and provides the basis for a numerical solution procedure. 

Let \begin{align*}
    D_1:=\{(t,s,x) \in [0,T] \times [0,T] \times \mathbb{R}\,|\,0 \leq s \leq t  \},\\
    D_2:= \{(t,s,x) \in [0,T] \times [0,\mathcal{S}] \times \mathbb{R}\,|\,0 \leq s \leq t  \},
\end{align*}
be the domains of the value function, depending on the initial class. 
\begin{thm}
    For $i \in \{1,2\}$, $V(i,t,s,x)$ is a viscosity solution to the system of equations
    \begin{align}
        &\begin{aligned}
        &\sup_{b\geq 0} \mathcal{A}^b v_1(t,s,x)=0,\ \qquad  t \in [0,T),\ s\in [0,t],\ x \in \mathbb{R}\label{eq:diffequation}\\&\sup_{b\geq 0} \mathcal{A}^b v_2(t,s,x)=0,\ \qquad t \in [0,T),\ s\in [0,t] \text{ and } s<\mathcal{S}, \ x \in \mathbb{R} 
        \end{aligned}\\
        &v_i(T,s,x)=h(x),\\
        &v_2(t,\mathcal{S},x)=v_1(t,0,x),
    \end{align}
    where $v_i: D_i \rightarrow \mathbb{R}$.
\end{thm}

\begin{proof}

We start by showing that $V$ is a supersolution. Let $i \in \{1,2\}$ and $(t,s,x) \in D_i$ be fixed. Let $\varphi_i$ be a smooth function with $\varphi_i(t,s,x)=V(i,t,s,x)$ and $\varphi_i\leq V(i,\cdot,\cdot,\cdot)$. We have to show that 
\begin{equation*}
    \sup_{b\geq 0} \mathcal{A}^b \varphi_i(t,s,x)\leq 0.
\end{equation*}

Let $t_S=t+(\mathcal{S}-s)-\vartheta$ and $T_\vartheta=T-\vartheta$, for some $\vartheta>0$ and define $\tau=T_1 \wedge t_S \wedge T_\theta \wedge t_h$, where $t_h=t+h$, for some $h>0$. Further, let $\tilde{b}$ be a constant control. Then, by Corollary~\ref{cor:DPP},
\begin{align*}
    \varphi_i(t,s,x)    &=V(i,t,s,x)\\
                        &=\sup_{b \in \mathcal{B}(t)} \mathbb{E}_{i,t,s,x}\left[V(I_\tau,\tau,S_\tau^b,X_\tau^b)   \right]\\
                        &\geq \mathbb{E}_{i,t,s,x}\big[V(I_{T_1},T_1,S_{T_1}^{\tilde{b}},X_{T_1}^{\tilde{b}})\ind{\{T_1\leq t_S\wedge t_h\}} \\
                        & \qquad + V(I_{t_s \wedge t_h},t_s \wedge t_h,\mathcal{S}^{\tilde{b}}_{t_s \wedge t_h},X^{\tilde{b}}_{t_s \wedge t_h})\ind{\{T_1 > t_S\wedge t_h\}} \big].
\end{align*}
By using that $V(i,\cdot,\cdot,\cdot)\geq \varphi_i$,
\begin{align*}
     &\varphi_i(t,s,x)  \geq \mathbb{E}_{i,t,s,x}\bigg[\big( \ind {\{Y_1 \leq \tilde{b} \}}\varphi_i\left(T_1,s+(T_1-t),x+\int_{s}^{s+T_1-t}(c-\pi_i(u,m_i))\diff u-Y_1\right) \\
     & + \ind {\{Y_1 > \tilde{b} \}} \varphi_2\left(T_1,0,x+\int_{s}^{s+T_1-t}(c-\pi_i(u,m_i))\diff u-r(Y_1,m_i)    \right) \ind{\{T_1\leq t_S\wedge t_h\}} \\
                        & + \varphi_i\left(t_s \wedge t_h,s+t_s \wedge t_h-t,x+\int_{s}^{s+t_s \wedge t_h-t}(c-\pi_i(u,m_i))\diff u\right)\ind{\{T_1 > t_S\wedge t_h\}} \bigg]\\
     &= \int_{0}^{t_S\wedge t_h-t} \lambda e^{-\lambda v} \Bigg(\int_0^{\tilde{b}} \varphi_i\left(t+v,s+v,x+\int_{s}^{s+v}(c-\pi_i(u,m_i))\diff u-y  \right)\diff F_Y(y) \\
     &+ \int_{\tilde{b}}^\infty \varphi_2\left(t+v,0,x+\int_{s}^{s+v}(c-\pi_i(u,m_i))\diff u-r(y,m_i)  \right)\diff F_Y(y)\Bigg) \diff v  \\ 
     &+ e^{-\lambda (t_s \wedge t_h-t)} \varphi_i\left(t_s \wedge t_h,s+t_s \wedge t_h-t,x+\int_{s}^{s+t_s \wedge t_h-t}(c-\pi_i(u,m_i))\diff u\right).
\end{align*}
Note that in the second line there is always $\varphi_2$, because one either switches to state 2 or stays there. The above is equivalent to 
\begin{align*}
    &0\geq e^{-\lambda (t_s \wedge t_h-t)} \Biggl(\varphi_i\left(t_s \wedge t_h,s+t_s \wedge t_h-t,x+\int_{s}^{s+t_s \wedge t_h-t}(c-\pi_i(u,m_i))\diff u\right)\\&\qquad -\varphi_i(t,s,x)\Biggr) + (e^{-\lambda (t_s \wedge t_h-t)} -1)\varphi_i(t,s,x)\\
    &+\int_{0}^{t_S\wedge t_h-t} \lambda e^{-\lambda v} \Bigg(\int_0^{\tilde{b}} \varphi_i\left(t+v,s+v,x+\int_{s}^{s+v}(c-\pi_i(u,m_i))\diff u-y  \right)\diff F_Y(y) \\
     &+ \int_{\tilde{b}}^\infty \varphi_2\left(t+v,0,x+\int_{s}^{s+v}(c-\pi_i(u,m_i))\diff u-r(y,m_i)  \right)\diff F_Y(y)\Bigg) \diff v.
\end{align*}

If we divide by $h$, send $h \rightarrow 0$ and since this holds for all values of $\tilde{b}$, we get that $\sup_{b\geq 0} \mathcal{A}^b \varphi(i,t,s,x)\leq 0$.\\
Now we show that $V$ is a subsolution. Let $\psi_i$ be a smooth function with\\ $\psi_i(t,s,x) ~=~V(i,t,s,x)$ and $\psi_i\geq V(i,\cdot,\cdot,\cdot)$. For notational simplicity, we write $\overline{x}=(t,s,x)$. We have to show, that $\sup_{b\geq 0} \mathcal{A}^b \psi_i(\overline{x})\geq 0$. We assume that 
\begin{equation*}
    \sup_{b\geq 0} \mathcal{A}^b \psi_i(\overline{x})\leq -\eta < 0,
\end{equation*} for some $\eta>0$ which means that $\mathcal{A}^b \psi_i(\overline{x}) \leq -\frac{\eta}{2}$ for all $b\geq 0$. Since $\psi_i$ is continuously differentiable, there is some $\delta>0$, such that for all $b\geq 0$,
\begin{equation*}
   \mathcal{A}^b \psi_i(\tilde{x})\leq -\frac{\eta}{2}, \qquad \forall \tilde{x}\in B_\delta(\overline{x}).
\end{equation*}
We define $t_\delta=t+\frac{\delta}{\sqrt{(2+c^2)}}$ and $t_T=T-\vartheta$ for some $\vartheta>0$ and consider $\tau = t_\delta \wedge t_T \wedge T_1$. Let $\varepsilon>0$. Then there is some $\varepsilon$-optimal control $b^\varepsilon$ such that
\begin{align*}
    \psi_i(\overline{x})=V(i,\overline{x}) &=\sup_{b \in \mathcal{B}(t)} \mathbb{E}_{i,\overline{x}}\left[ V(I_\tau,\tau,S_\tau^b,X_\tau^b) \ind{\{T>\tau\}} + u(X_T^b) \ind{\{T\leq\tau\}}\right]\\
    &\leq  \mathbb{E}_{i,\overline{x}}\left[ V(I_\tau,\tau,S_\tau^{b^{\varepsilon}},X_\tau^{b^{\varepsilon}})\right]+\varepsilon\\
    &\leq  \mathbb{E}_{i,\overline{x}}\left[ \psi_{I_\tau}(\tau,S_\tau^{b^{\varepsilon}},X_\tau^{b^{\varepsilon}})\right] + \varepsilon.
\end{align*}
Then, by Dynkin's formula,
\begin{align*}
    \psi_i(\overline{x}) \leq \psi_i(\overline{x}) + \mathbb{E}_{i,\overline{x}}\left[\int_t^\tau \mathcal{A}^{b^{\varepsilon}} \psi_{i}(v,S_v^{b^{\varepsilon}},X_v^{b^{\varepsilon}}) \diff v  \right] + \varepsilon.
\end{align*}
Therefore, $\mathbb{E}_{i,\overline{x}}\left[\int_t^\tau \mathcal{A}^{b^{\varepsilon}} \psi_{i}(v,S_v^{b^{\varepsilon}},X_v^{b^{\varepsilon}}) \diff v  \right] \geq - \varepsilon$. On the other hand,
\begin{align*}
    \mathbb{E}_{i,\overline{x}}\left[\int_t^\tau \mathcal{A}^{b^{\varepsilon}} \psi_{i}(v,S_v^{b^{\varepsilon}},X_v^{b^{\varepsilon}}) \diff v  \right] = \mathbb{E}_{i,\overline{x}}\left[\ind{\{T_1 \leq t_\delta \wedge t_T\}}\int_t^{T_1}\mathcal{A}^{b^{\varepsilon}} \psi_{i}(v,S_v^{b^{\varepsilon}},X_v^{b^{\varepsilon}}) \diff v  \right] \\
    + \mathbb{E}_{i,\overline{x}}\left[\ind{\{T_1 > t_\delta \wedge t_T\}}\int_t^{t_\delta \wedge t_T}\mathcal{A}^{b^{\varepsilon}} \psi_{i}(v,S_v^{b^{\varepsilon}},X_v^{b^{\varepsilon}}) \diff v  \right] \\
    \leq -\frac{\eta}{2} \left(\mathbb{E}[(T_1-t)\ind{\{T_1 \leq t_\delta \wedge t_T\}}] + e^{-\lambda(t_\delta \wedge t_T -t)}(t_\delta\wedge t_T-t) \right)\\
    = -\frac{\eta}{2 \lambda} (1-e^{-\lambda(t_\delta\wedge t_T-t)}).
\end{align*}
By choosing $\varepsilon<\frac{\eta}{2 \lambda} \left(1-e^{-\lambda(t_\delta\wedge t_T-t)}\right)$, and noting that the right-hand side is independent of $\varepsilon$, we arrive at a contradiction.
\end{proof}
To ensure uniqueness, which is essential for the application of a numerical scheme, we now prove that a comparison principle holds in our setting. The proof combines elements from \cite{AzcueMuler2014} and \cite{HuyenPham2009}, and we use additional techniques due to the specific nature of our problem.
\begin{thm}
    For $j \in \{1,2\}$, let $\tilde{\underline{v}}_j:D_j\to\mathbb{R}$ and $\tilde{\overline{v}}_j:D_j\to\mathbb{R}$ be Lipschitz-continuous viscosity sub- and supersolutions to \eqref{eq:diffequation} respectively. If $\tilde{\underline{v}}_j\leq\tilde{\overline{v}}_j$ on $\partial D_j$, then $\tilde{\underline{v}}_j\leq \tilde{\overline{v}}_j$ on $D_j$. 
\end{thm}

\begin{proof}
    For $j \in \{1,2\}$, let $\tilde{\underline{v}}_j$ and $\tilde{\overline{v}}_j$ be a viscosity sub- and supersolution respectively. We argue by contradiction and assume that there is a $k \in \{1,2\}$ and a point $(t_0,x_0,s_0) \in D_k$ such that
    \begin{equation*}
        \tilde{\underline{v}}_k(t_0,s_0,x_0)-\tilde{\overline{v}}_k(t_0,s_0,x_0)>0.
    \end{equation*}
    For $j \in \{1,2\}$, we define $\underline{v}_j(t,s,x)=e^{\delta t}\tilde{\underline{v}}_j(t,s,x)$ and $\overline{v}_j(t,s,x)=e^{\delta t}\tilde{\overline{v}}_j(t,s,x)$, with $\delta>0$. Then, these functions are viscosity sub- and supersolution respectively to the equation
    \begin{align}
    \label{proof:differentequation}
          \sup_{b\geq 0}&\left(\frac{\partial}{\partial t}+\frac{\partial}{\partial s}\right) v_j(t,s,x)+\left(c-\pi_j(s,m_j)\right)\frac{\partial}{\partial x}v_j(t,s,x) \nonumber + \lambda \int_{0}^{b}v_j(t,s,x-y)\diff F_Y(y) \nonumber \\
    &+ \lambda \int_{b}^{\infty}v_2(t,s,x-r(y,m_j))\diff F_Y(y) - (\lambda + \delta) v_j(t,s,x)=0. 
    \end{align}

    Let $\phi(t,x)=e^{-\gamma t}(1+x^2)$ and define $\overline{v}_j^\varepsilon(t,s,x)=\overline{v}_j(t,s,x)+\varepsilon  \phi(t,x)$. We show that there is a $\gamma>0$ such that  $\overline{v}_j^\varepsilon$ is again a supersolution to \eqref{proof:differentequation}. Therefore, for some $b\geq 0$, we consider
    \begin{align*}
       &\left(\frac{\partial}{\partial t}+\frac{\partial}{\partial s}\right) \phi(t,x)+\left(c-\pi_j(s,m_j)\right)\frac{\partial}{\partial x}\phi(t,x) \nonumber + \lambda \int_{0}^{b}\phi(t,x-y)\diff F_Y(y) \nonumber \\
    &+ \lambda \int_{b}^{\infty}\phi(t,x-r(y,m_j))\diff F_Y(y) - (\lambda+\delta) \phi(t,x)\\
    & =  e^{-\gamma t} \Big(-\gamma(1+x^2)+(c-\pi_i(s,m_j))2x + \lambda  \int_{0}^{b}(1+(x-y))^2\diff F_Y(y) \nonumber \\
    &+ \lambda \int_{b}^{\infty}(1+(x-r(y,m_j))^2)\diff F_Y(y) - (\lambda + \delta)(1+x^2)\Big)\\
    &\leq e^{-\gamma t}\Big((\lambda-\delta-\gamma)x^2 + 2(c-\pi_j(s,m_j)-\lambda\mathbb{E}[Y+r(Y,m_j)])x\\&\qquad+\lambda\mathbb{E}[Y^2+r(Y,m_j)^2]+(\lambda-\delta-\gamma) \Big).
    \end{align*}
    \\
    The expression is less than zero if the term inside the brackets is negative. Since this term is a polynomial in $x$, we can analyze its sign by studying its coefficients. If the leading coefficient is negative and the discriminant is non-positive, then the entire polynomial is non-positive for all $x$. Therefore, we need at least that $\gamma>\lambda-\delta$ and define

\begin{align*}
    &b_j(s)=2(c-\pi_j(s,m_j)-\lambda\mathbb{E}[Y+r(Y,m_j)]),\\
    &\tilde{c}_j = \lambda\mathbb{E}[Y^2+r(Y,m_j)^2].
\end{align*}
Then, the discriminant is given by
\begin{align*}
    b_j^2(s)-4(\lambda-\delta - \gamma)(\tilde{c}_j+(\lambda-\delta - \gamma)).
\end{align*}
and is less than 0 if we set 
\begin{align*}
    \gamma\geq  \max_{j \in \{1,2\}}\sup_{s \in [0,T]}\lambda-\delta + \frac{1}{2}\left(\tilde{c}_j+\sqrt{b_j(s)^2+\tilde{c}_j^2}\right).
\end{align*}
With this choice of $\gamma$ and since this holds for all $b\geq0$, we get that $\overline{v}_j^\varepsilon$ is again a supersolution to the corresponding equation. \\ \ \\
    Since $\underline{v}_j$ and $\overline{v}_j$ are Lipschitz continuous, there is a constant $L>0$ such that 
    \begin{align}
    \label{proof:growthcondition}
        \sup_{D_j}\frac{|\underline{v}_j(t,s,x)|+|\overline{v}_j(t,s,x)|}{1+|x|}<L.
    \end{align}
    Since 
    \begin{equation*}
        \underline{v}_j(t,s,x)-\overline{v}^\varepsilon_j(t,s,x)<L(1+|x|)-\varepsilon e^{-\gamma T} (1+x^2),
    \end{equation*}
    we can find a constant $a>0$ such that this is less than zero for $x>a$ and $x<-a$. Therefore, we can find a compact set $A_j \subseteq D_j$ such that
    \begin{equation*}
        M = \max_{j \in \{1,2\}} \sup_{A_j}(\underline{v}_j-\overline{v}_j^\varepsilon)= \max_{j \in \{1,2\}}\sup_{D_j}(\underline{v}_j-\overline{v}_j^\varepsilon)> \underline{v}_k(t_0,s_0,x_0)-\overline{v}_k^\varepsilon(t_0,s_0,x_0)>0,
    \end{equation*}
    for some small enough $\varepsilon>0$. We call the corresponding maximizing index $i$ and the maximizer $w^*=(t^*,s^*,x^*)$. One can show, that $\underline{v}_i$ and $\overline{v}_i^\varepsilon$ are still Lipschitz continuous on $A_i$ with constant $m$. \\
    We define,
    \begin{align*}
        g_\nu^i(t_1,s_1,x_1,t_2,s_2,x_2):=&\frac{\nu}{2}\left[(t_1-t_2)^2+(s_1-s_2)^2+(x_1-x_2)^2\right],
    \end{align*}
    for some $\nu>0$ and further
    \begin{equation*}
        G_\nu ^i (t_1,s_1,x_1,t_2,s_2,x_2) := \underline{v}_i(t_1,s_1,x_1)-\overline{v}^\varepsilon_i(t_2,s_2,x_2)- g_\nu^i(t_1,s_1,x_1,t_2,s_2,x_2).
    \end{equation*}
    Let $M_\nu ^i=\max_{A_i \times A_i} G_\nu^i$ and $w_\nu=(t_1^\nu,s_1^\nu,x_1^\nu,t_2^\nu,s_2^\nu,x_2^\nu)$ the corresponding maximizer. Then,
    \begin{equation*}
        M_\nu ^i \geq G_\nu^i(t^*,s^*,x^*,t^*,s^*,x^*)=M.
    \end{equation*}
    We need that $(t_1^\nu,s_1^\nu,x_1^\nu,t_2^\nu,s_2^\nu,x_2^\nu) \notin \partial (A_i \times A_i)$. 
    The boundaries are of the following form: 
   \begin{align*}
\partial (A_1 \times A_1) =  \big\{\, &(t_1, t_2, s_1, s_2, x_1, x_2) \in A_1 \ \big| \\
& t_j \in \{0,T\} \ \text{or} \ s_j \in \{0,T\} \ \text{or} \ s_j = t_j \ \text{or} \ x_j \in \{-a,a\} \text{ for } j \in \{1,2\}\big\}.
\end{align*}
\begin{align*}
\partial (A_2 \times A_2) =  \big\{\, &(t_1, t_2, s_1, s_2, x_1, x_2) \in A_2 \ \big| \\
& t_j \in \{0,T\} \ \text{or} \ s_j \in \{0,\mathcal{S}\} \ \text{or} \ s_j = t_j \ \text{or} \ x_j \in \{-a,a\} \text{ for } j \in \{1,2\}\big\}.
\end{align*}
    By \eqref{proof:growthcondition} and the construction of $A_i$, we know that
    \begin{equation*}
         G_\nu ^i (t_1,s_1,x_1,t_2,s_2,x_2) \leq 2L (1+a)- \frac{\nu}{2}\left[(t_1-t_2)^2+(s_1-s_2)^2+(x_1-x_2)^2\right].
    \end{equation*}
    If either $|t_1-t_2|>\vartheta$, $|s_1-s_2|>\vartheta$ or $|x_1-x_2|>\vartheta$ for some $\vartheta>0$, then this is negative for $\nu > \frac{4L(1+a)}{\vartheta}$. \\So let us consider the case where $t_1=t_2$, $s_1=s_2$ and $x_1=x_2$.
    Then,
    \begin{align*}
             G_\nu ^i (t_1,s_1,x_1,t_1,s_1,x_1) =  \underline{v}_i(t_1,s_1,x_1)-\overline{v}^\varepsilon_i(t_1,s_1,x_1).
    \end{align*}
    In the cases where $t_1=T$,  $s_1=\mathcal{S}$ (if $i=2$), $t_1=0$, $s_1=0$ or $t_1=s_1$, this is less than or equal to zero by the theorem's assumption. \\If $x_1 \in \{-a,a\}$, this is less than or equal to zero by construction of $A_i$. Therefore, the maximum is not attained on the boundary. \\ \ \\
    \noindent
    Since $G_\nu^i$ attains a maximum in $w_\nu$, it holds that for any $(t_1,s_1,x_1) \in D_i$, 
    \begin{equation*}
        G_\nu^i(t_1,s_1,x_1,t_2^\nu,s_2^\nu,x_2^\nu)\leq G_\nu^i(w_\nu),
    \end{equation*}
    and therefore
    \begin{equation*}
        \underline{v}_i(t_1,s_1,x_1)-\underline{v}_i(t_1^\nu,s_1^\nu,x_1^\nu)\leq g_\nu^i(w_\nu)-g_\nu^i (t_1,s_1,x_1,t_2^\nu,s_2^\nu,x_2^\nu).
    \end{equation*}
    Together with Taylor's formula we arrive at
    \begin{align*}&\limsup_{(t_1,s_1,x_1)\rightarrow (t_1^\nu,s_1^\nu,x_1^\nu)}\biggl(\frac{ \underline{v}_i(t_1,s_1,x_1)-\underline{v}_i(t_1^\nu,s_1^\nu,x_1^\nu)}{|t_1^\nu-t_1|+|s_1^\nu-s_1|+|x_1^\nu-x_1|}\\
        &-\frac{ (t_1^\nu-t_1)\frac{\partial}{\partial t_1}g_\nu(w_\nu)+(s_1^\nu-s_1)\frac{\partial}{\partial s_1}g_\nu(w_\nu)+(x_1^\nu-x_1)\frac{\partial}{\partial x_1}g_\nu(w_\nu)}{|t_1^\nu-t_1|+|s_1^\nu-s_1|+|x_1^\nu-x_1|}\biggr)\leq 0
    \end{align*}
    and therefore $\partial_{t_1+s_1+x_1} g_\nu^i(w_\nu) \in D^+ (\underline{v}_i)(t_1^\nu,s_1^\nu,x_1^\nu)$, which is the set of all superdifferentials of $\underline{v}_i$ at $(t_1^\nu,s_1^\nu,x_1^\nu)$.
    With a similar argument,
     \begin{align*}
        &\liminf_{(t_2,s_2,x_2)\rightarrow (t_2^\nu,s_2^\nu,x_2^\nu)} \biggl(\frac{ \overline{v}^\varepsilon_i(t_2,s_2,x_2)-\overline{v}^\varepsilon_i(t_2^\nu,s_2^\nu,x_2^\nu)}{|t_2^\nu-t_2|+|s_2^\nu-s_2|+|x_2^\nu-x_2|}\\
        &-\frac{(t_2^\nu-t_2)(-\frac{\partial}{\partial t_2}g_\nu(w_\nu))+(s_2^\nu-s_2)(-\frac{\partial}{\partial s_2}g_\nu(w_\nu))+(x_2^\nu-x_2)(-\frac{\partial}{\partial x_2}g_\nu(w_\nu))}{|t_2^\nu-t_2|+|s_2^\nu-s_2|+|x_2^\nu-x_2|}\biggr)\geq 0.
    \end{align*}
    Therefore, $-\partial_{t_2+s_2+x_2} g_\nu^i(w_\nu) \in D^- (\overline{v}^\varepsilon_i)(t_2^\nu,s_2^\nu,x_2^\nu)$, which is the set of all subdifferentials of $\overline{v}^\varepsilon_i$ in $(t_2^\nu,s_2^\nu,x_2^\nu)$. In addition, 
    \begin{equation}
    \label{eq:g_x=-g_y}
        \partial_{t_1+s_1+x_1} g_\nu^i(w_\nu)=-\partial_{t_2+s_2+x_2} g_\nu^i(w_\nu).
    \end{equation}
    
    \noindent
    Now we return to the equation \eqref{proof:differentequation}. There is a $b\geq 0$ such that
    \begin{align}
    \label{eq:proof>=0}
        &\left(\frac{\partial}{\partial t_1}+\frac{\partial}{\partial s_1}\right) g_{\nu}^i(t_1,s_1,x_1,t_2,s_2,x_2)|_{w_{\nu}} \nonumber \\&+\left(c-\pi_i(s_1^\nu,m_i)\right)\frac{\partial}{\partial x_1}g_{\nu}^i(t_1,s_1,x_1,t_2,s_2,x_2)|_{w_{\nu}} \nonumber \\
    &+ \lambda \left(\int_{0}^{b}\underline{v}_i(t_1^\nu,s_1^\nu,x_1^\nu-y)\diff F_Y(y) \nonumber 
    +\int_{b}^{\infty}\underline{v}_2(t_1^\nu,0,x_1^\nu-r(y,m_i))\diff F_Y(y)\right) \\&- (\lambda+\delta) \underline{v}_i(t_1^\nu,s_1^\nu,x_1^\nu) \geq 0,
    \end{align}
    and 
    \begin{align}
    \label{eq:proof<=0}
        &-\left(\frac{\partial}{\partial t_2}+\frac{\partial}{\partial s_2}\right) g_{\nu}^i(t_1,s_1,x_1,t_2,s_2,x_2)|_{w_{\nu}}\\
        &-\left(c-\pi_i(s_2^\nu,m_i)\right)\frac{\partial}{\partial x_2}g_{\nu}^i(t_1,s_1,x_1,t_2,s_2,x_2)|_{w_{\nu}} \nonumber \\
    &+ \lambda \left(\int_{0}^{b}\overline{v}^\varepsilon_i(t_2^\nu,s_2^\nu,x_2^\nu-y)\diff F_Y(y) \nonumber 
    +\int_{b}^{\infty}\overline{v}^\varepsilon_2(t_2^\nu,0,x_2^\nu-r(y,m_i))\diff F_Y(y)\right) \\&- (\lambda+\delta) \overline{v}^\varepsilon_i(t_2^\nu,s_2^\nu,x_2^\nu) \leq 0,
    \end{align}
    By subtracting the left hand side of \eqref{eq:proof>=0} from the LHS of \eqref{eq:proof<=0} together with \eqref{eq:g_x=-g_y}, we arrive at
  
    \begin{align}
    \label{proof:(delta+lambda)}
         &(\lambda + \delta)(\underline{v}_i(t_1^\nu,s_1^\nu,x_1^\nu)-\overline{v}^\varepsilon_i(t_2^\nu,s_2^\nu,x_2^\nu)) \nonumber\\
        &\leq \left(c-\pi_i(s_2^\nu,m_i)\right)\frac{\partial}{\partial x_2}g_{\nu}^i(t_1,s_1,x_1,t_2,s_2,x_2)|_{w_{\nu}} \nonumber \\&+ \left(c-\pi_i(s_1^\nu,m_i)\right)\frac{\partial}{\partial x_1}g_{\nu}^i(t_1,s_1,x_1,t_2,s_2,x_2)|_{w_{\nu}}\nonumber\\
        &+\lambda \left(\int_{0}^{b}\underline{v}_i(t_1^\nu,s_1^\nu,x_1^\nu-y)\diff F_Y(y)  
   +\int_{b}^{\infty}\underline{v}_2(t_1^\nu,0,x_1^\nu-r(y,m_i))\diff F_Y(y)\right) \nonumber\\
   &-  \lambda \left(\int_{0}^{b}\overline{v}^\varepsilon_i(t_2^\nu,s_2^\nu,x_2^\nu-y)\diff F_Y(y)  
    +\int_{b}^{\infty}\overline{v}^\varepsilon_2(t_2^\nu,0,x_2^\nu-r(y,m_i))\diff F_Y(y)\right)
    \end{align}
    Since $w_\nu$ is the maximum of $G_\nu$ on $A_i \times A_i$,  we have that 
    \begin{align*}
        G_\nu ^i (t_1^\nu,s_1^\nu,x_1^\nu,t_1^\nu,s_1^\nu,x_1^\nu)+G_\nu ^i (t_2^\nu,s_2^\nu,x_2^\nu,t_2^\nu,s_2^\nu,x_2^\nu) \leq 2G_\nu(w_\nu).
    \end{align*}
    This is equivalent to
    \begin{align*}
        &\underline{v}_i(t_1^\nu,s_1^\nu,x_1^\nu)-\overline{v}_i^\varepsilon(t_1^\nu,s_1^\nu,x_1^\nu)+ \underline{v}_i(t_2^\nu,s_2^\nu,x_2^\nu)-\overline{v}_i^\varepsilon(t_2^\nu,s_2^\nu,x_2^\nu) \\
        &\leq 2(\underline{v}_i(t_1^\nu,s_1^\nu,x_1^\nu)-\overline{v}_i^\varepsilon(t_2^\nu,s_2^\nu,x_2^\nu)) - \nu ||(t_1^\nu,s_1^\nu,x_1^\nu)-(t_2^\nu,s_2^\nu,x_2^\nu)||^2_2
    \end{align*}
    and further 
    \begin{align*}
       &\nu||(t_1^\nu,s_1^\nu,x_1^\nu)-(t_2^\nu,s_2^\nu,x_2^\nu)||^2_2 \\
       &\leq \underline{v}_i(t_1^\nu,s_1^\nu,x_1^\nu)- \underline{v}_i(t_2^\nu,s_2^\nu,x_2^\nu)+ \overline{v}_i^\varepsilon(t_1^\nu,s_1^\nu,x_1^\nu)-\overline{v}_i^\varepsilon(t_2^\nu,s_2^\nu,x_2^\nu)\\
       & \leq m ||(t_1^\nu,s_1^\nu,x_1^\nu)-(t_2^\nu,s_2^\nu,x_2^\nu)||_1
       \leq m \sqrt{2} ||(t_1^\nu,s_1^\nu,x_1^\nu)-(t_2^\nu,s_2^\nu,x_2^\nu)||_2.
    \end{align*}
    Thus,
    \begin{equation}
    \label{proof:inequality}
        ||(t_1^\nu,s_1^\nu,x_1^\nu)-(t_2^\nu,s_2^\nu,x_2^\nu)||_2 \leq \frac{\sqrt{2}m}{\nu}.
    \end{equation}
    So if we take a sequence $(\nu_n)_{n \in \mathbb{N}}$ such that $w_\nu$ converges to $(\overline{t}_1,\overline{s}_1,\overline{x}_1,\overline{t}_2,\overline{s}_2,\overline{x}_2)$ as $\nu_n \rightarrow \infty$, then we get by \eqref{proof:inequality} that $(\overline{t}_1,\overline{s}_1,\overline{x}_1)=(\overline{t}_2,\overline{s}_2,\overline{x}_2)$.

By \eqref{proof:(delta+lambda)} and \eqref{eq:g_x=-g_y}, we arrive at
   \begin{align*}
         &(\lambda + \delta)(\underline{v}_i(\overline{t}_1,\overline{s}_1,\overline{x}_1)-\overline{v}^\varepsilon_i(\overline{t}_1,\overline{s}_1,\overline{x}_1)) \nonumber\\
        &\leq \lambda \left(\int_{0}^{b}\underline{v}_i(\overline{t}_1,\overline{s}_1,\overline{x}_1-y)\diff F_Y(y)  
   +\int_{b}^{\infty}\underline{v}_2(\overline{t}_1,0,\overline{x}_1-r(y,m_i))\diff F_Y(y)\right) \nonumber\\
   &-  \lambda \left(\int_{0}^{b}\overline{v}^\varepsilon_i(\overline{t}_1,\overline{s}_1,\overline{x}_1-y)\diff F_Y(y)  
    +\int_{b}^{\infty}\overline{v}^\varepsilon_2(\overline{t}_1,0,\overline{x}_1-r(y,m_i))\diff F_Y(y)\right) \leq \lambda M.
    \end{align*}
    So 
    \begin{equation*}
       M \leq \lim_{n \to \infty}M_{\nu_n} =\underline{v}_i(\overline{t}_1,\overline{s}_1,\overline{x}_1)-\overline{v}^\varepsilon_i(\overline{t}_1,\overline{s}_1,\overline{x}_1)\leq \frac{\lambda}{\lambda + \delta} M,
    \end{equation*}
    which is a contradiction and thus completes the proof.

\end{proof}

We have now established that $V$ is the unique viscosity solution to \eqref{eq:diffequation} which allows us to proceed with a numerical approach.

\section{Numerical example}
\label{sec:Numerics}
Based on the results from the previous section, we can construct the value function by finding a solution to the system of Hamilton–Jacobi–Bellman equations. We will carry out this process numerically, following an approach inspired by the the paper \textit{Permanent health insurance} by \cite{DavisHealthInsurance}. In contrast to their PDMP-PDE system, our setting involves a PIDE system, i.e., a PDE with additional integral terms.

The following derivations also work for Markovian controls of the form 
\begin{equation*}
   b(t)=\tilde{b}(I_t,t,S_t^b,X_t^b),
\end{equation*} but for simplicity of notation, we assume for the moment that $b \geq 0$ is constant. Consider the functional
\begin{equation*}
    J(i,t,s,x,b)=\mathbb{E}_{i,t,s,x}[h(X^b_T)].
\end{equation*}
The idea is the following: Let $\{\sigma_n\}_{n \in \mathbb{N}}$ be the sequence of all jump times, i.e., a combination of the jump times $\{T_j\}_{j \in \mathbb{N}}$ and the transition times from state 2 to state 1. Further, let $\sigma_0=0$ and $\tilde{N}_t$ be the number of jumps in $(0,t]$. Additionally, we say that $\{\tilde{Y}_j\}_{j \in \mathbb{N}}$ are the corresponding jump heights, with $\tilde{Y}_j=0$ for all jumps which are not triggered by the Poisson process. For $n \in \mathbb{N}_0$, we define
\begin{equation*}
    v_i^n(t,s,x)=\mathbb{E}_{i,t,s,x}[h(X^b_T)\ind{\{\tilde{N}_T-\tilde{N}_t\leq n\}}]=\sum_{k=0}^n \mathbb{E}_{i,t,s,x}[h(X^b_T)\ind{\{\tilde{N}_T-\tilde{N}_t=k\}}].
\end{equation*}

Then, it holds that,
\begin{align*}
    v_1^{k+1}(t,s,x)=&\mathbb{E}_{1,t,s,x} \biggl[h(X_T)\mathbbm{1}_{\{T<\sigma_{\tilde{N}_t+1}\}} \\&+\biggl(v_1^k(\sigma_{\tilde{N}_t+1},s+\sigma_{\tilde{N}_t+1}-t,X_{\sigma_{\tilde{N}_t+1}})\ind{\{\tilde{Y}_{\tilde{N}_t+1}<b\}}  \\&+v_2^k(\sigma_{\tilde{N}_t+1},0,X_{\sigma_{\tilde{N}_t+1}})  \ind{\{\tilde{Y}_{\tilde{N}_t+1}\geq b\}}  \biggr) \mathbbm{1}_{\{T\geq \sigma_{\tilde{N}_t+1}\}} \biggr],
\end{align*}
and
\begin{align*}
    v_2^{k+1}(t,s,x)=&\mathbb{E}_{2,t,s,x} \biggl[h(X_T)\mathbbm{1}_{\{T<\sigma_{\tilde{N}_t+1}\}} \\&+v_1^k(\sigma_{\tilde{N}_t+1},0,X_{\sigma_{\tilde{N}_t+1}}) \ind{\{ T\geq \sigma_{\tilde{N}_t+1}>t+\mathcal{S}-s \}}\\
    &+ \biggl(v_2^k(\sigma_{\tilde{N}_t+1},s+\sigma_{\tilde{N}_t+1}-t,X_{\sigma_{\tilde{N}_t+1}})\ind{\{\tilde{Y}_{\tilde{N}_t+1}<b\}} \\&+v_2^k(\sigma_{\tilde{N}_t+1},0,X_{\sigma_{\tilde{N}_t+1}})  \ind{\{\tilde{Y}_{\tilde{N}_t+1}\geq b\}}  \biggr) \mathbbm{1}_{\{\sigma_{\tilde{N}_t+1} \leq T\wedge t+\mathcal{S}-s \}} \biggr].
\end{align*}
Moreover,
\begin{align*}
    &v_1^0(t,s,x)=h\left(x+\int_{s}^{s+(T-t)}c-\pi_1(s,m_1)\diff s\right)e^{-\lambda(T-t)},\\
     &v_2^0(t,s,x)=h\left(x+\int_{s}^{s+(T-t)}c-\pi_2(s,m_2)\diff s\right) e^{-\lambda(T-t)} \ind{\{s+T-t< \mathcal{S}\}}.
\end{align*}

It holds that
\begin{equation*}
    \lim _ {n \to \infty} v_i^n(t,x) =J(i,t,s,x,b),
\end{equation*} since $\mathbb{E}_{i,t,s,x}[|h(X_T^b)|]<\infty$. 

\newpage
For a fixed $b\geq 0$, the functions $v_1^n(t,s,x)$ and $v_2^n(t,s,x)$ (provided differentiability) solve the following system of differential equations:
\begin{align*}
      &\left(\frac{\partial}{\partial t}+\frac{\partial}{\partial s}\right) v_1^{n+1}(t,s,x)+\left(c-\pi_1(s,m_1)\right)\frac{\partial}{\partial x}v_1^{n+1}(t,s,x) \nonumber \\
    &+ \lambda \int_{0}^{b}v_1^n(t,s,x-y)\diff F_Y(y) \nonumber \\
    &+ \lambda \int_{b}^{\infty}v_2^n(t,0,x-r(y,m_1))\diff F_Y(y) - \lambda v_1^{n+1}(t,s,x)=0,\\ \ \\
    &\left(\frac{\partial}{\partial t}+\frac{\partial}{\partial s}\right) v_2^{n+1}(t,s,x)+\left(c-\pi_2(s,m_2)\right)\frac{\partial}{\partial x}v_2^{n+1}(t,s,x) \nonumber \\
     &+ \lambda \int_{0}^{b}v_2^{n}(t,s,x-y)\diff F_Y(y)  \nonumber \\
     &+ \lambda \int_{b}^{\infty}v_2^n(t,0,x-r(y,m_2))\diff F_Y(y)- \lambda v_2^{n+1}(t,s,x)=0,
\end{align*}
with boundary conditions 
\begin{align*}
    &v_i^n(T,s,x)=h(x),\\
        &v_2^{n+1}(t,\mathcal{S},x)=v_1^n(t,0,x).
\end{align*}
For a similar iterative approach, again in the context of permanent health insurance, see \citep{Rolski}.
Together with a policy iteration algorithm (see for example \citep{KushnerDupuis}), we can compute an approximation of 
\begin{equation*}
    V(i,t,s,x)=\sup_{b \in \mathcal{B}(t)}J(i,t,s,x,b).
\end{equation*}

We now examine a concrete example. Let $h(x)=\max\{-10^{10},-e^{-\gamma x}\}$, where $\gamma>0$. This mimics an exponential utility function. Further, let the claim sizes be exponentially distributed with mean $\mu>0$. We specify the model parameters as in Table \ref{tab:Model_specifications}.

\begin{table}[h]
\centering
\caption{Model specifications}
\label{tab:Model_specifications}
\setlength{\tabcolsep}{10pt}
\begin{tabular}{c c c c c c c c c c c}
\toprule
$T$ & $\mathcal{S}$ & $\lambda$ & $\mu$ & $m_1$ & $m_2$ & $\pi_1(s,m_1)$ & $\pi_2(s,m_2)$ & $c$ & $\gamma$ \\
\midrule
5 & 2 & 1 & 1 & 0 & 0 & $-\frac{7s}{10 T} + 1$ & 1.1 & 1.2 & 0.5 \\
\bottomrule
\end{tabular}
\end{table}
\ \\ 
Since technically, $x \in \mathbb{R}$, we only solve the equation on $[0,5]$. We denote the approximated value function by $\tilde{V}$. After 5 iterations, we arrive at Figures \ref{fig:approxValueFunctionsinx}, \ref{fig:approxValueFunctionsint} and \ref{fig:approxValueFunctionsins}. 

\begin{figure}[h]

\includegraphics[scale=1]{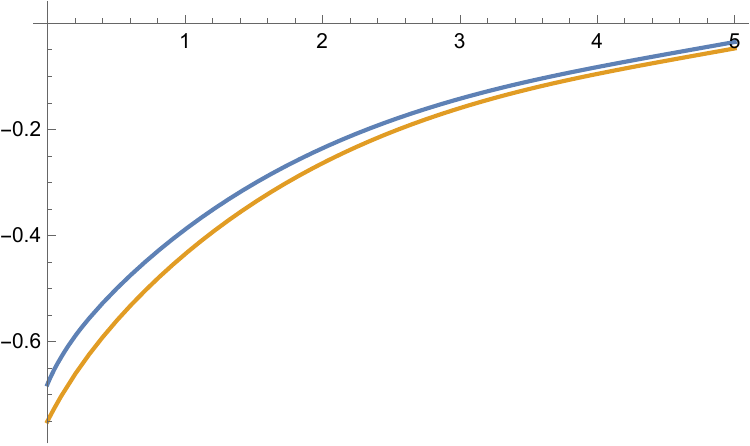}
        \caption{Approximated value functions $x \mapsto \tilde{V}(1,0,0,x)$ (blue) and $x \mapsto \tilde{V}(2,0,0,x)$ (orange). }
        \label{fig:approxValueFunctionsinx}
\end{figure}

\begin{figure}[H]
\includegraphics[scale=1]{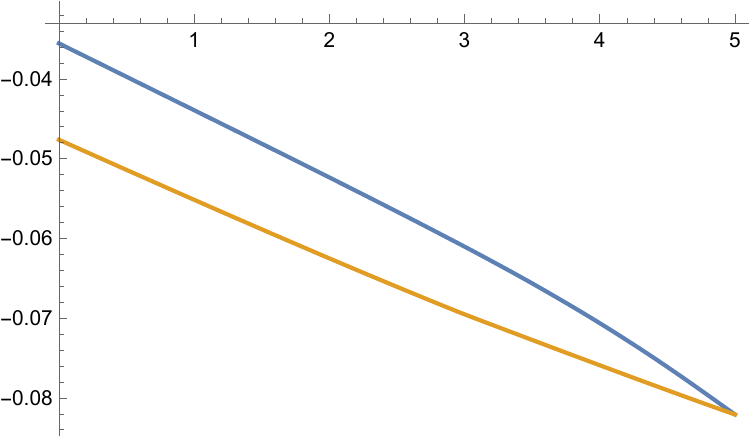}
        \caption{Approximated value functions $t \mapsto \tilde{V}(1,t,0,5)$ (blue) and $t \mapsto \tilde{V}(2,t,0,5)$ (orange). }
        \label{fig:approxValueFunctionsint}
\end{figure}

\begin{figure}[h]
\includegraphics[scale=1]{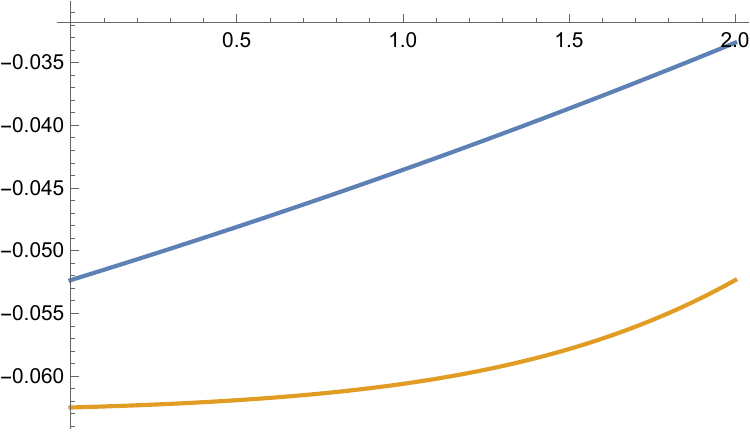}
        \caption{Approximated value functions $s \mapsto \tilde{V}(1,\mathcal{S},s,5)$ (blue) and $s \mapsto \tilde{V}(2,\mathcal{S},s,5)$ (orange). }
        \label{fig:approxValueFunctionsins}
\end{figure}

The corresponding approximate optimal barriers are functions $\tilde{b}(1,t,s,x)$ and $\tilde{b}(2,t,s,x)$. In Figure \ref{fig:approxbarrier1}, we see that in state 1, the barrier strategy $\tilde{b}$ decreases when $t$ is close to maturity $T$, for fixed values of $s$ and $x$. The closer $t$ is to maturity, the more likely a claim is to be reported. This suggests, that the disadvantage of switching to class 2, namely the higher premium rate, becomes less relevant. Very small claims remain unreported, even near to maturity. In Figure \ref{fig:approxbarrier2}, we see that something similar holds in state~2, if the time since reporting the last claim is nearly $\mathcal{S}$. If $t=s=1.6$, then it is possible to reach $\mathcal{S}$ and therefore get transferred to the the better premium class $C_1$ before maturity, which is why only claims above $0.15$ get reported. As $t$ approaches maturity, however, more claims are reported, since upgrading to a better premium class either becomes unprofitable relative to the claim sizes or cannot be achieved in the remaining time.

\begin{figure}[ht]
    \centering
    \begin{subfigure}{0.48\textwidth}
        \includegraphics[width=\linewidth]{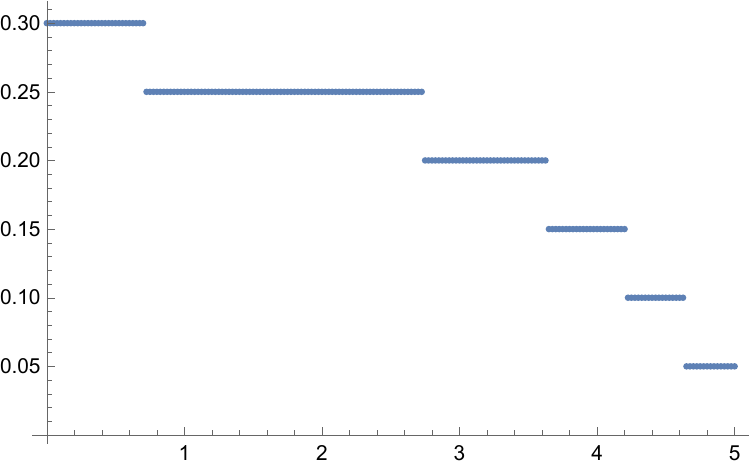}
        \caption{$t \mapsto \tilde{b}(1,t,0,5)$.}
        \label{fig:approxbarrier1}
    \end{subfigure}
    \hfill
    \begin{subfigure}{0.48\textwidth}
        \includegraphics[width=\linewidth]{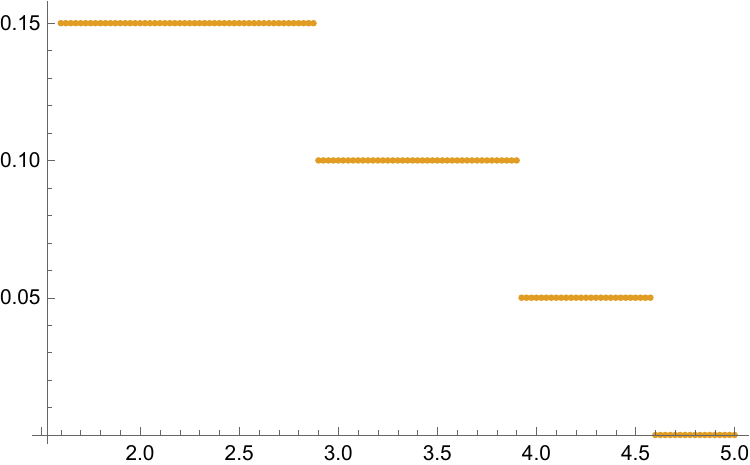}
        \caption{$t \mapsto \tilde{b}(2,t,1.6,5)$.}
        \label{fig:approxbarrier2}
    \end{subfigure}
    \caption{Approximated barrier strategies.}
    \label{fig:approxbarrier_both}
\end{figure}

\begin{remark}
    We observe in Figures \ref{fig:approxbarrier1} and \ref{fig:approxbarrier2} that the approximated optimal barrier is piecewise constant. In fact, the difference between consecutive values corresponds exactly to a single spatial step size $h_x$. This behavior arises from the numerical solution of the PDE, where one must carefully account for the relationship between the time step sizes $h_t=h_s$ and the spatial step size $h_x$. Consequently, a smaller spatial step size $h_x$ necessitates a proportionally smaller time step $h_t$, leading to more costly and time-intensive computations. Otherwise, oscillations which perturb the results may occur. Further refinements of the numerical scheme are desirable to enhance both accuracy and efficiency.
\end{remark}

\section*{Acknowledgments}
The authors thank Julia Eisenberg and Jean-Fran\c{c}ois Renaud for carefully reading the manuscript and for their valuable input.

\section*{Funding}

This research was funded in whole or in part by the Austrian Science Fund (FWF) [10.55776/P33317]. For open access purposes, the authors have applied a CC BY public copyright license to any author-accepted manuscript version arising from this submission.

\bibliographystyle{apalike}


\begin{thebibliography}{}

\bibitem[Azcue and Muler, 2014]{AzcueMuler2014}
Azcue, P. and Muler, N. (2014).
\newblock {\em Stochastic optimization in insurance}.
\newblock SpringerBriefs in Quantitative Finance. Springer, New York.
\newblock A dynamic programming approach.

\bibitem[Cao et~al., 2025a]{Young2025_meanvariance}
Cao, J., Li, D., Young, V.~R., and Zou, B. (2025a).
\newblock Continuous-time optimal reporting with full insurance under the mean-variance criterion.
\newblock {\em Insurance Math. Econom.}, 120:79--90.

\bibitem[Cao et~al., 2025b]{Young2025}
Cao, J., Li, D., Young, V.~R., and Zou, B. (2025b).
\newblock Optimal loss reporting in continuous time with full insurance.
\newblock {\em SIAM J. Financial Math.}, 16(2):448--479.

\bibitem[Charpentier et~al., 2017]{charpentier2017}
Charpentier, A., David, A., and Elie, R. (2017).
\newblock Optimal claiming strategies in bonus malus systems and implied markov chains.
\newblock {\em Risks}, 5(4).

\bibitem[Davis, 1993]{Davis}
Davis, M. H.~A. (1993).
\newblock {\em Markov models and optimization}.
\newblock Monographs on statistics and applied probability. Chapman \& Hall, London.

\bibitem[Davis and Vellekoop, 1995]{DavisHealthInsurance}
Davis, M. H.~A. and Vellekoop, M.~H. (1995).
\newblock Permanent health insurance: a case study in piecewise deterministic markov modelling.
\newblock {\em Mitteilungen der Schweizerischen Vereinigung der Versicherungsmathematiker}, pages 177--212.

\bibitem[De~Pril, 1979]{DePril1979}
De~Pril, N. (1979).
\newblock Optimal claim decisions for a bonus-malus system: a continuous approach.
\newblock {\em Astin Bull.}, 10(2):215--222.

\bibitem[Fleming and Soner, 2006]{FlemingSoner2006}
Fleming, W.~H. and Soner, H.~M. (2006).
\newblock {\em Controlled {M}arkov processes and viscosity solutions}, volume~25 of {\em Stochastic Modelling and Applied Probability}.
\newblock Springer, New York, second edition.

\bibitem[Haehling~von Lanzenauer, 1974]{Lanzenauer1974}
Haehling~von Lanzenauer, C. (1974).
\newblock Optimal claim decisions by policyholders in automobile insurance with merit-rating structures.
\newblock {\em Operations Research}, 22(5):979--990.

\bibitem[Kushner and Dupuis, 2001]{KushnerDupuis}
Kushner, H.~J. and Dupuis, P. (2001).
\newblock {\em Numerical methods for stochastic control problems in continuous time {\textup{(2nd ed.)}}}, volume~24 of {\em Applications of Mathematics}.
\newblock Springer, New York.

\bibitem[Lemaire, 1998]{Lemaire.1998}
Lemaire, J. (1998).
\newblock Bonus-malus systems.
\newblock {\em North American Actuarial Journal}, 2(1):26--38.

\bibitem[Pham, 2009]{HuyenPham2009}
Pham, H. (2009).
\newblock {\em Continuous-time stochastic control and optimization with financial applications}, volume~61 of {\em Stochastic Modelling and Applied Probability}.
\newblock Springer-Verlag, Berlin.

\bibitem[Rolski et~al., 1999]{Rolski}
Rolski, T., Schmidli, H., Schmidt, V., and Teugels, J. (1999).
\newblock {\em Stochastic processes for insurance and finance}.
\newblock Wiley Series in Probability and Statistics. John Wiley \& Sons, Ltd., Chichester.

\bibitem[Schmidli, 2008]{Schmidli2008}
Schmidli, H. (2008).
\newblock {\em Stochastic control in insurance}.
\newblock Probability and its Applications. Springer-Verlag, London.

\bibitem[Straub, 1969]{Straub}
Straub, E. (1969).
\newblock {Zur Theorie der Prämienstufensysteme}.
\newblock {\em Mitteilungen der Schweizerischen Vereinigung der Versicherungsmathematiker}, pages 75--85.

\bibitem[Zacks and Levikson, 2004]{ZacksLevikson2004}
Zacks, S. and Levikson, B. (2004).
\newblock Claiming strategies and premium levels for bonus malus systems.
\newblock {\em Scand. Actuar. J.}, 2004(1):14--27.

\end{thebibliography}
\end{document}